\documentclass[11pt]{amsart}

\usepackage{epigamath}

\usepackage[english]{babel}
\usepackage{url}

\numberwithin{equation}{section}

\usepackage{enumitem}
\usepackage{amssymb,amscd,amsthm,amsmath}
\usepackage{mathtools}
\usepackage{url}

\newtheorem{theorem}{Theorem}[section]

\newtheorem{lemma}[theorem]{Lemma}
\newtheorem{conjecture}[theorem]{Conjecture}
\newtheorem{corollary}[theorem]{Corollary}
\newtheorem{proposition}[theorem]{Proposition}

\theoremstyle{definition}
\newtheorem{definition}[theorem]{Definition}

\theoremstyle{remark}
\newtheorem{remark}[theorem]{Remark}
\newtheorem{example}[theorem]{Example}
\newtheorem{setting}[theorem]{Main Setting}

\setlist[enumerate,1]{label={\rm(\arabic*)}, ref={\rm\arabic*}}

\newcommand{\cI}{\mathcal{I}}
\newcommand{\cO}{\mathcal{O}}

\renewcommand{\P}{\mathbb{P}}

\DeclareMathOperator{\Sym}{Sym}
\DeclareMathOperator{\rk}{rk}
\DeclareMathOperator{\OG}{OG}
\DeclareMathOperator{\SG}{SG}

\newcommand{\longhookrightarrow}{\lhook\joinrel\longrightarrow}
\newcommand{\lra}{\longrightarrow}

\EpigaVolumeYear{10}{2026} \EpigaArticleNr{5} \ReceivedOn{October 21, 2024}
\InFinalFormOn{July 28, 2025}
\AcceptedOn{October 8, 2025}

\title{Algebraic hyperbolicity of very general hypersurfaces in homogeneous varieties}
\titlemark{Algebraic hyperbolicity of very general hypersurfaces in homogeneous varieties}

\author{Lucas Mioranci}
\address{Instituto de Matem\'atica Pura e Aplicada, Estrada Dona Castorina 110, Jardim Bot\^anico, 22460-320 Rio de Janeiro, RJ, Brazil}
\email{lucas.mioranci@impa.br}

\authormark{L.~Mioranci}

\AbstractInEnglish{We generalize techniques by Coskun, Riedl, and Yeong, and obtain an almost optimal bound on the degree for the algebraic hyperbolicity of very general hypersurfaces in rational homogeneous varieties. As examples, we work out the cases of very general hypersurfaces in Grassmannians and products therefore, orthogonal and symplectic Grassmannians, and flag varieties.}


\acknowledgement{The author is supported by CNPq Post-doctoral fellowship 152039/2024-4.}

\begin{document}



\maketitle

\begin{prelims}

\DisplayAbstractInEnglish

\bigskip

\DisplayKeyWords

\medskip

\DisplayMSCclass

\end{prelims}


\newpage

\setcounter{tocdepth}{1}

\tableofcontents


\section{Introduction}

A complex projective variety $X$ is \textit{algebraically hyperbolic} if, for some ample divisor $H$, there exists a real number $\epsilon > 0$ such that, for any integral curve $C\subset X$, the inequality
\[ 2g(C)-2\ge \epsilon \deg_H(C) \]
is satisfied, where $g(C)$ is the geometric genus of $C$. The goal of this paper is to prove the algebraic hyperbolicity of very general hypersurfaces in homogeneous varieties by generalizing techniques developed by Clemens \cite{Cl86, Cl03}, Ein \cite{Ei88, Ei91}, Voisin \cite{Vo96,Vo98}, Pacienza \cite{Pa03,Pa04}, Clemens and Ran \cite{ClR04}, Coskun and Riedl \cite{CR19a,CR19}, and Yeong \cite{Y22}.

A complex manifold $X$ is called (\textit{Brody}) \textit{hyperbolic} if every entire map $f\colon \mathbb{C}\to X$ is constant, and it is \textit{Kobayashi hyperbolic} if the Kobayashi pseudometric is non-degenerate. In general, Kobayashi hyperbolicity implies Brody hyperbolicity, and when $X$ is a compact complex manifold, Brody proved that the converse holds; see \cite{Br78}. Hyperbolicity is conjectured to have deep connections to the geometry and arithmetic of varieties. 

\begin{conjecture}[Lang's conjectures \cite{L86}]
  \leavevmode
    \begin{itemize} 
    \item A projective algebraic variety $X$ is hyperbolic if and only if every subvariety of $X$ is of general type.
    \item A projective variety is hyperbolic if and only if it contains only finitely many rational points over any finite field extension.
\end{itemize}
\end{conjecture}

\begin{conjecture}[Green--Griffiths--Lang conjecture \cite{GG80}]
    A smooth projective variety $X$ of general type contains a proper subvariety $S\subsetneq X$ containing all entire curves of $X$.
\end{conjecture}

See Demailly's survey paper \cite{De20} for an exposition on hyperbolicity.

It can, however, be challenging to prove the hyperbolicity of a given variety. Demailly introduced algebraic hyperbolicity as an algebraic analogue for hyperbolicity, proved that, for smooth projective varieties, hyperbolicity implies algebraic hyperbolicity, and conjectured the converse; see \cite{De95}.

A lot of progress has been made in the study of the algebraic hyperbolicity of very general hypersurfaces in $\P^n$. For $n=3$, Xu \cite{Xu94} proved that very general hypersurfaces $X\subset \P^3$ of degree $d\ge 6$ are algebraically hyperbolic. This was improved by Coskun and Riedl \cite{CR19a}, who showed that a very general quintic surface is also algebraically hyperbolic. Since surfaces of degree at most~$4$ contain rational curves and cannot be algebraically hyperbolic, the classification for very general hypersurfaces in $\P^3$ is complete. For $n\ge 4$, Clemens and Ein \cite{Cl86, Ei88} proved that $X$ is algebraically hyperbolic when $d\ge 2n$. This was improved to $d\ge 2n-1$ by Voisin \cite{Vo96, Vo98}. Pacienza \cite{Pa04} and Clemens and Ran \cite{ClR04} proved it for $d\ge 2n-2$ and $n\ge 6$. Yeong \cite{Y22} did the octic 4-fold case and improved the result to $d\ge 2n-2$ and $n\ge 5$. If $d\le 2n-3$, then $X$ contains lines, so it is not algebraically hyperbolic. Hence, the only remaining case in $\P^n$ is for sextic threefolds.

Haase and Ilten \cite{HI21} started to classify algebraically hyperbolic surfaces in toric threefolds, and gave almost sharp bounds for the cases $\P^2\times \P^1$, $\P^1\times \P^1\times \P^1$, the blowup of $\P^3$ at a point, and the weighted projective space $\P(1,1,1,n)$. Coskun and Riedl \cite{CR19} streamlined their scroll method and used a careful analysis of where various line bundles fail to be section dominating to obtain sharp bounds in the cases $\P^2\times \P^1$, $\P^1\times \P^1\times \P^1$, the blowup of $\P^3$ at a point, and $\mathbb{F}_e\times \P^1$, and partial bounds for very general surfaces in other threefolds. Robins \cite{Rob23} used Haase and Ilten's method to determine many algebraically hyperbolic surfaces in toric varieties with Picard rank 2 and 3. Yeong \cite{Y22} applied Coskun and Riedl's techniques to almost completely classify algebraically hyperbolic hypersurfaces in $\P^m\times \P^n$, except for some degrees in $\P^3\times \P^1$.

In this paper, we generalize the techniques of Coskun, Riedl, and Yeong to the context of more general homogeneous varieties. As an application, we obtain an almost sharp bound for the degree for which very general hypersurfaces of (products of) Grassmannians, orthogonal Grassmannians, symplectic Grassmannians, and flag varieties are algebraically hyperbolic.

Let us start by defining the general setting we will refer to throughout the paper.

\begin{setting}\label{main_setting}
    Let $A$ be a smooth complex projective variety such that
    \begin{itemize}
        \item $A$ is a rational homogeneous variety with a transitive action by an algebraic group $G$;
        \item $A$ has an embedding into a product of projective spaces
        \[ i\colon A\longhookrightarrow \P^{N_1}\times \cdots \times \P^{N_m} \]
        such that the divisors $H_1, \hdots , H_m$ of $A$ corresponding to the pullbacks of the hyperplane classes $\cO_{\P^{N_i}}(1)$ generate the Picard group of $A$;
        \item the embedding $i$ is projectively normal; that is, the restriction
        \[ H^0\left (\P^{N_1}\times \cdots \times \P^{N_m}, \cO_{\P^{N_1}\times \cdots \times \P^{N_m}}(d_1,\ldots ,d_m)\right )\lra H^0(A,d_1H_1+\cdots +d_mH_m) \]
        is surjective for all $d_1,\ldots,d_m > 0$.
    \end{itemize}

    Write the class of the canonical divisor of $A$ as $K_A = a_1H_1 + \cdots + a_mH_m$. Let $\mathcal{E}$ be the line bundle $d_1H_1 + \cdots + d_mH_m$ for $d_1,\hdots ,d_m >0$, equivariant under $G$, and assume that $H_1, \hdots , H_m$ is a collection of section-dominating line bundles for $\mathcal{E}$ (see Definition~\ref{definition_section_dominating}). A degree $(d_1, \hdots , d_m)$ hypersurface $X$ of $A$ is a section of $\mathcal{E}$.
\end{setting}

\begin{theorem}\label{main_theorem} Let $A$ be a rational homogeneous variety as in Main Setting~\ref{main_setting}.
    \begin{itemize}
        \item If $\dim A\ge 4$ and $d_i\ge \dim A-a_i-2$ for all $1\le i\le m$, then a very general hypersurface $X$ of degree $(d_1, \hdots, d_m)$ is algebraically hyperbolic.
        \item If $d_i\le \dim A - a_i - 4$ for some $1\le i\le m$, then a general hypersurface $X$ of degree $(d_1, \hdots, d_m)$ contains lines. In particular, $X$ is not algebraically hyperbolic.
    \end{itemize}
\end{theorem}

The cases with $d_i\ge \dim A - a_i - 3$ for all $1\le i\le m$ and $d_i = \dim A-a_i-3$ for some $1\le i\le m$ remain open. When $\dim A\le 4$, there are examples when it fails to imply algebraic hyperbolicity (see Examples~\ref{Example_2_2} and~\ref{Example_2_1_1}.) For $A=\P^4$, the case of sextic threefolds remains open. However, when $\dim A\ge 5$, it does imply algebraic hyperbolicity for $A=\P^n$ and $A=\P^m\times \P^n$; see \cite{Y22}. We conjecture $X$ is algebraically hyperbolic in these cases when $\dim A\ge 5$ (see Conjecture~\ref{dim_5_conjecture}.)

We then apply the theorem to concrete examples of homogeneous varieties. We note that the theorem can also be applied to products of these varieties.

{\samepage
\begin{theorem} Let $A$ be a
    \begin{enumerate}[itemsep=8pt]
    
        \item \textbf{Grassmannian $A=G(k,n)$}

        \begin{itemize}[itemsep = 4pt]
            \item If\, $\dim A=k(n-k)\ge 4$ and $d\ge k(n-k)+n-2$, then a very general degree $d$ hypersurface of $G(k,n)$ is algebraically hyperbolic.
            \item If\, $d\le k(n-k)+n-4$, then a general degree $d$ hypersurface of $G(k,n)$ contains a line. In particular, it is not algebraically hyperbolic.
        \end{itemize}

        \item \textbf{Product of Grassmannians $A = \prod_{i=1}^m G(k_i,n_i)$}
        \begin{itemize}[itemsep = 4pt]
            \item If\, $\dim A = \sum_{i=1}^m k_i(n_i-k_i)\ge 4$ and $d_i\ge \left (\sum_{j=1}^m k_j(n_j-k_j)\right ) + n_i - 2$ for all $1\le i\le m$, then a very general hypersurface of\, $\prod_{i=1}^m G(k_i,n_i)$ of degree $(d_1,\hdots,d_m)$ is algebraically hyperbolic.

            \item If\, $d_i\le \left (\sum_{j=1}^m k_j(n_j-k_j)\right ) + n_i - 4$ for some $i$, then a general hypersurface of degree $(d_1,\hdots ,d_m)$ of $\prod_{i=1}^m G(k_i,n_i)$ contains a line. In particular, it is not algebraically hyperbolic.
        \end{itemize}

        \item \textbf{Orthogonal Grassmannian $A=\OG(k,n)$}
        \begin{itemize}[itemsep = 4pt]
            \item If\, $\dim \OG(k,n)=\frac{k(2n-3k-1)}{2}\ge 4$ and $d\ge \frac{k(2n-3k-1)}{2}+n-k-3$, then a very general hypersurface of $\OG(k,n)$ of degree $d$ is algebraically hyperbolic.

            \item If\, $d\le \frac{k(2n-3k-1)}{2}+n-k-5$, then a general hypersurface of degree $d$ of\, $\OG(k,n)$ contains a line. In particular, it is not algebraically hyperbolic.
        \end{itemize}
        
        \item \textbf{Symplectic Grassmannian $A=\SG(k,n)$}
        \begin{itemize}[itemsep = 4pt]
            \item If\, $\dim A=\frac{k(2n-3k+1)}{2}\ge 4$ and $d\ge \frac{k(2n-3k+1)}{2}+n-k-1$, then a very general degree $d$ hypersurface of\, $\SG(k,n)$ is algebraically hyperbolic.

            \item If\, $d\le \frac{k(2n-3k+1)}{2}+n-k-3$, then a general degree $d$ hypersurface of\, $\SG(k,n)$ contains a line. In particular, it is not algebraically hyperbolic.
        \end{itemize}
        
        \item \textbf{Flag Variety $A=F(k_1, \hdots, k_m; n)$}
        \begin{itemize}[itemsep = 4pt]
            \item If\, $\dim A = \sum_{i=1}^{m}k_{i}(k_{i+1}-k_i)\ge 4$ and $d_{i}\ge \left (\sum_{j=1}^{m}k_{j}(k_{j+1}-k_j)\right ) + k_{i+1} - k_{i-1} - 2$ for all $1\le i\le m$, then a very general hypersurface of\, $F(k_1, \hdots, k_m; n)$ of degree $(d_1, \hdots, d_m)$ is algebraically hyperbolic.

            \item If\, $d_{i}\le \left (\sum_{j=1}^{m}k_{j}(k_{j+1}-k_j)\right ) + k_{i+1} - k_{i-1} - 4$ for some $i$, then a general hypersurface of\, $F(k_1, \hdots, k_m; n)$ of degree $(d_1, \hdots, d_m)$ contains a line. In particular, it is not algebraically hyperbolic.
        \end{itemize}
        
    \end{enumerate}
    
\end{theorem}
}

\subsection*{Organization of the paper} In Section~\ref{sec2}, we describe the general setup we work on and obtain a first, weaker bound on algebraic hyperbolicity. Section~\ref{sec3} is dedicated to working out the scroll argument to rational homogeneous varieties and in proving the algebraic hyperbolicity statement in Theorem~\ref{main_theorem}. In Section~\ref{sec4}, we determine when general hypersurfaces in rational homogeneous varieties contain lines, and conclude the non-algebraically hyperbolic cases from Theorem~\ref{main_theorem}. In Section~\ref{sec5}, we apply the bound obtained in the previous sections to specific examples of homogeneous varieties.

\subsection*{Acknowledgements} This manuscript is part of my Ph.D.\ thesis at the University of Illinois at Chicago. I would like to thank my advisor, Izzet Coskun, for our weekly meetings and support in writing this paper. I am grateful to Eric Riedl, Wern Yeong, and Jackson Morrow for their valuable advice and suggestions. I also want to thank the journal referees for identifying some mistakes and suggesting improvements.

\section{Preliminaries}\label{sec2}

\subsection{Setup}\label{setup}
We first describe the general setup for the techniques developed by Clemens and Ein \cite{Ei88, Ei91}, Voisin \cite{Vo96, Vo98}, Pacienza \cite{Pa03, Pa04}, Coskun and Riedl \cite{CR19a, CR19}, and Yeong \cite{Y22}.

Let $A$ be a smooth, complex projective variety of dimension $D$, and assume $A$ admits a transitive group action by an algebraic group $G$. Let $\mathcal{E}$ a globally generated vector bundle on $A$, equivariant under $G$, of rank $r<D-1$, and let $V = H^0(A,\mathcal{E})$. We will be particularly interested in the case when $\mathcal{E}$ is a very ample line bundle on $A$.

Let $X$ be the zero locus of a very general section of $\mathcal{E}$, and suppose $X$ contains a curve $Y$ of degree $e$ and geometric genus $g$. Then, if $\mathcal{X}_1$ is the universal hypersurface over $V$, we get the relative Hilbert scheme $\mathcal{H}\to V$ with universal curve $\mathcal{Y}_1$, where the general fiber of $\mathcal{Y}_1\to \mathcal{H}$ is a curve of geometric genus $g$ and degree $e$. We can find a $G$-invariant subvariety $U\subset \mathcal{H}$ such that the map $U\to V$ is \'etale. Restricting $\mathcal{Y}_1$ to $U$ and taking a resolution of the general fiber, we get a smooth family $\mathcal{Y}\to U$ whose fibers are smooth curves of genus $g$. We pull $\mathcal{X}_1$ back to a family $\mathcal{X}$ over $U$, with projection maps $\pi_1\colon \mathcal{X}\to U$ and $\pi_2\colon \mathcal{X}\to A$. There is a natural generically injective map $h\colon \mathcal{Y}\to \mathcal{X}$.

We define the \textit{vertical tangent sheaf} $T_{\mathcal{X}/A}$ by the short exact sequence
\[ 0\longrightarrow T_{\mathcal{X}/A}\longrightarrow T_{\mathcal{X}} \longrightarrow \pi_2^*T_A \longrightarrow 0. \]

Since we constructed $\mathcal{Y}$ to be stable under the $G$-action, we have that $\pi_2\circ h$ dominates $A$ and that the map $T_{\mathcal{Y}}\to h^*\pi_2^*T_A$ is surjective. Similarly, we define the vertical tangent sheaf $T_{\mathcal{Y}/A}$ as the kernel in the short exact sequence
\[ 0\longrightarrow T_{\mathcal{Y}/A}\longrightarrow T_{\mathcal{Y}} \longrightarrow h^*\pi_2^*T_A \longrightarrow 0. \]

Let $M_{\mathcal{E}}$ denote the \textit{Lazarsfeld--Mukai bundle} associated to $\mathcal{E}$, defined by the sequence
\begin{equation*}
0\longrightarrow M_{\mathcal{E}} \longrightarrow H^0(A,\mathcal{E})\otimes \cO_A \longrightarrow \mathcal{E} \longrightarrow 0.
\end{equation*}

For a general element $t$ of $U$, let $Y_t$ be the fiber of $\mathcal{Y}$ over $t$ and $X_t$ be the fiber of $\mathcal{X}$ over $t$. Denote by $h_t\colon Y_t\to X_t$ the restriction of $h$ to $Y_t$. The following result translates the relation between the vertical tangent sheaves and the Lazarsfeld--Mukai bundle.

\begin{proposition}[\textit{cf.} {\cite[Proposition 2.1]{CR19}}]\label{CR19Prop2.1}\leavevmode
\begin{enumerate}
    \item $N_{h_t/X_t}\cong N_{h/\mathcal{X}}|_{Y_t}$.
    \item $T_{\mathcal{X}/A}\cong \pi_2^* M_{\mathcal{E}}$.
    \item $N_{h/\mathcal{X}}$ is the cokernel of the map of vertical tangent sheaves $T_{\mathcal{Y}/A}\to h^*T_{\mathcal{X}/A}$.
\end{enumerate}
\end{proposition}

\subsection{The degree of the normal bundle and section-dominating line bundles}

There is a close relation between the degree of the normal bundle $N_{h_t/X_t}$ and the genus of the curve $Y_t$ described by the following lemma. We can use it to bound the genus of $Y_t$ by finding a bound for the degree of the normal bundle $N_{h_t/X_t}$.

\begin{lemma}[\textit{cf.} {\cite[Lemma 2.2]{CR19}}] \label{deggenus}
The following holds:
\[ \deg \left(N_{h_t/X_t}\right) = 2g(Y_t)-2 - K_{X_t}\cdot h_t(Y_t). \]
\end{lemma}

By Proposition~\ref{CR19Prop2.1}, we have a surjection
\begin{equation}\label{surjection_from_M_E}
    M_{\mathcal{E}}|_{Y_t}\lra N_{h_t/X_t}.
\end{equation}

We can get a surjection to $N_{h_t/X_t}$ from a simpler bundle by using the notion of \textit{section-dominating} line bundles.

\begin{definition}[\textit{cf.} {\cite[Definition 2.3]{CR19}}]\label{definition_section_dominating}
Let $\mathcal{E}$ be a vector bundle on $A$. A collection of non-trivial globally generated line bundles $L_1, \hdots , L_m$ is a \textit{section-dominating collection} of line bundles for $\mathcal{E}$ if $\mathcal{E}\otimes L_i^\vee$ is globally generated for every $1\le i\le m$ and the map
\[ \bigoplus_{i=1}^{m} \left ( H^0(L_i\otimes \cI_p)\otimes H^0\left(\mathcal{E}\otimes L_i^\vee\right) \right )\lra H^0(\mathcal{E}\otimes \cI_p) \]
is surjective for every point $p\in A$.
\end{definition}

\begin{example}
Suppose $A = \P^n$, and let $H$ be the hyperplane class. Then, for any $\mathcal{E} = dH$ with $d>0$, $H$ is a section-dominating line bundle for $\mathcal{E}$. To prove this, we need to show that
\[ H^0\left(\cO_{\P^n}(1)\otimes \cI_p\right)\otimes H^0\left(\cO_{\P^n}(d-1)\right)\lra H^0\left(\cO_{\P^n}(d)\otimes \cI_p\right) \]
is surjective. Choose coordinates $x_0, \hdots , x_n$ for $\P^n$ so that $p=(1:0:\cdots:0)$. Then
$H^0(\cO_{\P^n}(d)\otimes \cI_p)$ is the set of degree $d$ forms without $x_0^d$ term, and these can be written as a combination of products of linear forms without $x_0$ term in $H^0(\cO_{\P^n}(1)\otimes \cI_p)$ and degree $d-1$ forms in $H^0(\cO_{\P^n}(d-1))$.
\end{example}

\begin{example}\label{section-dominating}
Let $A = G(k,n)$ be the Grassmannians of $k$-planes in the affine $n$-space, and let $H$ correspond to the hyperplane class, that is, the pullback $p^*\cO_{\P^N}(1)$ via the Pl\"ucker embedding $p\colon G(k,n)\to \P^N$. Then, for any $\mathcal{E} = dH$ with $d>0$, $H$ is a section-dominating line bundle for $\mathcal{E}$. This follows directly from the previous example and the projective normality of the Grassmannian under the Pl\"ucker embedding (\textit{cf.} \cite{RR85}.) By \cite{RR85}, a similar result is true for flag varieties and Schubert varieties.
\end{example}

\begin{example}\label{section-dom-products}
This property can be extended to products: Let $A = \prod_{i=1}^m A_i$ be a product of varieties $A_i$ such that, for each $i$, $H_i$ is a section-dominating line bundle for $\mathcal{E}_i=d_iH_i$ for some $d_i>0$. Then $H_1, \hdots , H_m$ is a section-dominating collection of line bundles for $\mathcal{E} = \sum_{i=1}^m d_iH_i$. 
\end{example}

Given a section of $H^0(\mathcal{E}\otimes L_i^\vee)$, the natural multiplication map defines maps $L_i\to \mathcal{E}$ and $H^0(L_i)\to H^0(\mathcal{E})$, and induces a map $M_{L_i}\to M_{\mathcal{E}}$. By choosing bases for $H^0(\mathcal{E}\otimes L_i^\vee)$, we can get a surjection to $M_{\mathcal{E}}$.

\begin{proposition}[\textit{cf.} {\cite[Proposition 2.7]{CR19}}]\label{surjectionofmukai}
Let $\mathcal{E}$ be a globally generated vector bundle and $M_{\mathcal{E}}$ the Lazarsfeld--Mukai bundle associated to $\mathcal{E}$. Let $L_1, \hdots , L_m$ be a section-dominating collection of line bundles for $\mathcal{E}$. Then, there is a surjection
\[ \bigoplus _{i=1}^{m}M_{L_i}^{\oplus s_i}\lra M_{\mathcal{E}} \]
for some integers $s_i$.
\end{proposition}

From Proposition~\ref{surjectionofmukai} and the surjection (\ref{surjection_from_M_E}), we obtain a surjection to the normal bundle
\begin{equation}\label{surjection_beta}
\beta\colon \bigoplus _{i=1}^{m}M_{L_i}^{\oplus s_i}|_{Y_t}\lra N_{h_t/X_t}.
\end{equation}

By choosing generic sections of $H^0(\mathcal{E}\otimes L_i^\vee)$ for the multiplication in Proposition~\ref{surjectionofmukai}, the rank of the image of $\beta$ is a strictly increasing function of the $s_i$ until surjectivity is reached (\textit{cf.} \cite[Section 2.2]{Cl03}.) In particular, we can choose a surjection $\beta$ so that $s_1 + \cdots + s_m\le \rk N_{h_t/X_t}$.

This surjection can be used to get a bound for the degree of $N_{h_t/X_t}$ in terms of the integers $s_i$.

\begin{lemma}\label{surjectioninequality}
Let $\beta\colon \bigoplus _{i=1}^{m}M_{L_i}^{\oplus s_i}|_{Y_t}\to N_{h_t/X_t}$ be a surjective map as in \eqref{surjection_beta}. Then
\[ \deg N_{h_t/X_t}\ge -\sum_{i=1}^{m} s_i\deg L_i|_{Y_t}. \]
\end{lemma}
\begin{proof}
Let $K$ be the kernel of $\beta$:
\[ 0\longrightarrow K\longrightarrow \bigoplus _{i=1}^{m}M_{L_i}^{\oplus s_i}|_{Y_t}\overset{\beta}{\longrightarrow} N_{h_t/X_t}\longrightarrow 0. \]
Then by the definition of the $M_{L_i}$, we get an injection $K\to \bigoplus _{i=1}^{m}H^0(A,L_i)^{\oplus s_i}\otimes \cO_A|_{Y_t}$. Since it is a direct sum of trivial bundles, $\bigoplus _{i=1}^{m}H^0(A,L_i)^{\oplus s_i}\otimes \cO_A|_{Y_t}$ is a semi-stable vector bundle. It follows that
\[ \frac{\deg K}{\rk K} = \mu(K)\le 0, \]
so
\[ \deg (K)\le 0; \]
hence, from the sequence above and the sequence defining $M_{L_i}$, we have
\begin{equation*} \pushQED{\qed}
  \deg N_{h_t/X_t} = \deg \left (\bigoplus _{i=1}^{m}M_{L_i}^{\oplus s_i}|_{Y_t}\right ) - \deg K \ge \deg \left (\bigoplus _{i=1}^{m}M_{L_i}^{\oplus s_i}|_{Y_t}\right ) = -\sum_{i=1}^m s_i\deg L_i|_{Y_t}.\qedhere \popQED
  \end{equation*}
\renewcommand{\qed}{}   
\end{proof}

By combining Lemmas~\ref{deggenus} and~\ref{surjectioninequality}, we obtain a bound for the genus:
\begin{equation}\label{genusboundwiths}
     2g(Y_t)-2\ge K_{X_t}\cdot Y_t - \sum_{i=1}^{m} s_i\deg L_i|_{Y_t}.
\end{equation}
Therefore, if we can control the integers $s_i$, we can prove the algebraic hyperbolicity of $X_t$. This leads us to our first criterion for hyperbolicity. This result is similar to \cite[Corollary~2.9]{CR19} and \cite[Theorem~3.6]{HI21}.

\begin{proposition}\label{method_1}
    Let $A$ be a homogeneous variety of dimension $D$, and let $\mathcal{E}$ be a very ample line bundle on $A$ as in Setup~\ref{setup}. Let $L_1, \hdots, L_m$ be a collection of section-dominating line bundles for $\mathcal{E}$ on $A$. If, for some $\epsilon > 0$,
    \[ K_{X_t}\cdot Y_t\ge (D-2+\epsilon)\cdot \sum_{i=1}^m\deg L_i|_{Y_t}, \]
    then $X_t$ is algebraically hyperbolic.
\end{proposition}
\begin{proof}
    As before, we can choose the $s_i$ so that $s_1+\cdots +s_m\le \rk N_{h_t/X_t} = D-2$. In particular, we can assume $s_i\le D-2$ for every $1\le i\le m$. Thus, by Inequality (\ref{genusboundwiths}), we get
    \[ 2g(Y_t)-2\ge K_{X_t}\cdot Y_t - (D-2)\sum_{i=1}^{m} \deg L_i|_{Y_t}.\]
    
    Therefore, if $K_{X_t}\cdot Y_t\ge (D-2+\epsilon)\cdot \sum_{i=1}^m\deg L_i|_{Y_t}$, we get
    \[ 2g(Y_t)-2\ge \epsilon\cdot \sum_{i=1}^m\deg L_i|_{Y_t}. \]
    
    Hence, $X_t$ is algebraically hyperbolic.
\end{proof}

\begin{remark}
    This initial setup can also be defined when $A$ is not homogeneous but contains a Zariski-open homogeneous set $A_0$. Results equivalent to Proposition~\ref{method_1} can be obtained in this case; see, for example, \cite[Theorem 1.2]{CR19}. This allowed Coskun and Riedl \cite{CR19} to classify algebraically hyperbolic surfaces $X$ in the cases when $A$ is a Hirzebruch surface $\P(\cO_{\P^1}\oplus \cO_{\P^1}(e))$, $A$ is $\P^3$ blown up at a single point, and $A$ is a weighted projective space $\P(1,1,1,n)$. We believe the generalizations made in this paper can also be worked out when $A$ contains a Zariski-open homogeneous set $A_0$, allowing future studies of higher-dimensional toric varieties, such as blowups of projective spaces and weighted projective spaces.
\end{remark}

\section{Scroll method for homogeneous varieties}\label{sec3}

In this section, we work out the scroll method introduced by Coskun and Riedl in \cite{CR19} in a more general setting. This technique builds a surface scroll on the curve $Y_t$ whose degree can be used to improve the bound for algebraic hyperbolicity obtained in Proposition~\ref{method_1}.

Here we will work on rational homogeneous varieties satisfying the Main Setting~\ref{main_setting}, which we recall here:

 Let $A$ be a smooth complex projective variety of dimension $D$ such that
    \begin{itemize}
        \item $A$ is a rational homogeneous variety with a transitive action by an algebraic group $G$;
        \item $A$ has a projectively normal embedding into a product of projective spaces $\P^{N_1}\times \cdots \times \P^{N_m}$;
        \item the divisors $H_1, \hdots , H_m$ of $A$ corresponding to the pullbacks of the hyperplane classes $\cO_{\P^{N_i}}(1)$ generate the Picard group of $A$.
    \end{itemize}

    Let $\mathcal{E}$ be the line bundle $d_1H_1 + \cdots + d_mH_m$ for $d_1,\hdots ,d_m >0$, equivariant under $G$, and assume that $H_1, \hdots , H_m$ is a collection of section-dominating line bundles for $\mathcal{E}$. To simplify the notation, we write $X$ for the zero locus of a very general section of $\mathcal{E}$, and $C$ for the curve in $X$, instead of $X_t$ and $Y_t$.
    
    Write the class of the canonical divisor of $A$ as $K_A = a_1H_1 + \cdots + a_mH_m$. By the adjunction formula, $K_X = \sum_{i=1}^m (a_i+d_i)H_i|_X$. In this case, Proposition~\ref{method_1} can be rephrased as follows.

\begin{corollary}\label{cor_method_1}
If there exists an $\epsilon > 0$ such that $d_i\ge (D-a_i-2) + \epsilon$ for all $1\le i\le m$, then $X$ is algebraically hyperbolic.
\end{corollary}
\begin{proof}
    Since 
    \[ K_{X}\cdot C = \sum_{i=1}^m (a_i+d_i)H_i\cdot C\ge (D-2+\epsilon)\cdot \sum_{i=1}^m H_i\cdot C, \]
    the result follows from Proposition~\ref{method_1}.
\end{proof}

The main theorem in this section shows that we do not need the $\epsilon$ in the bound above. 

\begin{definition}
    Given a product of projective spaces $\P^{N_1}\times \cdots \times \P^{N_m}$, the fibers of the projection
    \[ \pi_1'\colon \P^{N_1}\times \cdots \times \P^{N_m}\lra \P^{N_2}\times \cdots \times \P^{N_m} \]
    are $\P^{N_1}$. We define a \textit{$\P^{N_1}$-line} as a line contained in one of these fibers. Equivalently, a $\P^{N_1}$-line is a curve of numerical class $H_1^{N_1-1}H_2^{N_2}\cdots H_m^{N_m}$. We define \textit{$\P^{N_i}$-lines} similarly for $i=1,\hdots , m$. We call a surface scroll a \textit{$\P^{N_i}$-scroll} if it contains a $\P^{N_i}$-line through every point.
\end{definition}

We can see the space of $\P^{N_1}$-lines as the product 
$G(2,N_1+1)\times \P^{N_2}\times \cdots \times \P^{N_m}$ and the Fano scheme of $\P^{N_1}$-lines contained in $A$ as a subscheme $F_1^{N_1}(A)\subset G(2,N_1+1)\times \P^{N_2}\times \cdots \times \P^{N_m}$, and similarly for $\P^{N_i}$-lines. We will need the dimension of the space of $\P^{N_i}$-lines in $A$.

\begin{lemma}\label{dim_lines_A}
    Let $F_1^{N_i}(A)$ be the Fano scheme of\, $\P^{N_i}$-lines in $A$. Then 
    \[ \dim F_1^{N_i}(A) = \dim A - a_i - 3. \]
\end{lemma}
\begin{proof}
    Let $L$ be a $\P^{N_i}$-line in $A$. Since $A$ is a rational homogeneous variety, its tangent bundle $T_A$ is globally generated; thus $T_A|_L$ is globally generated. A classic deformation argument  (\textit{cf.} \cite[Theorem 2.6]{HDAG}) then shows that $\dim F_1^{N_i}(A) = h^0(L, N_{L/A}) = \dim A - K_A\cdot L - 3 = \dim A - a_i - 3$.
\end{proof}

We use the following lemma to construct a scroll.

\begin{lemma}[\textit{cf.} {\cite[Lemma 2.13]{CR19}}]\label{scroll}
A rank~$1$ quotient $Q$ of\, $M_{H_i}|_C$ induces a $\P^{N_i}$-scroll $\Sigma$ over $C$ of\, $H_i$-degree equal to $\deg Q + (H_i\cdot C)$.
\end{lemma}

\begin{theorem}\label{scroll_method}
Let $A$ be a rational homogeneous variety as in Main Setting~\ref{main_setting}. If $\dim A\ge 4$ and $d_i\ge \dim A-a_i-2$ for all $1\le i\le m$, then $X$ is algebraically hyperbolic.
\end{theorem}
\begin{proof}
By Proposition~\ref{surjectionofmukai}, there is a surjection 
\[ \beta \colon \bigoplus_{i=1}^{m} M_{H_i}^{\oplus s_i}|_C\lra N_{C/X}. \]
Without loss of generality, assume $s_1\ge \cdots \ge s_m$. As before, we can assume $s_1 + \cdots + s_m\le \rk N_{C/X} = D - 2$. We divide the proof into the cases $s_1\le D - 3$ and $s_1 = D - 2$.

\textbf{Case 1.} First suppose that $s_1\le D - 3$. So, $s_i\le D-3$ for all $i$. Then, the inequality  
\[ 2g(C)-2\ge K_{X}\cdot C - \sum_{i=1}^{m} s_i\deg H_i|_{C} \]
from (\ref{genusboundwiths}) implies that
\[ 2g(C)-2\ge \left (\sum_{i=1}^m (a_i+d_i)H_i\right )\cdot C - (D-3)\sum_{i=1}^{m} (H_i\cdot C), \]
or equivalently
\[ 2g(C)-2\ge \sum_{i=1}^m (a_i+d_i - D + 3)(H_i\cdot C). \]
Since, by hypothesis, $d_i\ge D-a_i-2$ for $i=1, \hdots , m$, we get
\[ 2g(C)-2\ge \sum_{i=1}^m (H_i\cdot C). \]

\textbf{Case 2.} Now suppose  instead that $s_1 = D-2$ is the minimum possible value for $s_1$ for which $\beta$ is surjective. So we have $s_2=\cdots =s_m=0$ and a surjection
\[ \beta \colon M_{H_1}^{\oplus D-2}|_C \lra N_{C/X}. \]

Since this is the minimum value for $s_1$ such that $\beta$ is surjective, this means that the induced map
\begin{equation}\label{quot_map}
    M_{H_1}|_{C} \lra N_{C/X} / \beta \left (M_{H_1}^{\oplus D-3}|_C\right )
\end{equation}
has a rank~$1$ image $Q$ (\textit{cf.} \cite[Section 2.2]{Cl03}.) By Lemma~\ref{scroll}, this induces a $\P^{N_1}$-scroll $\Sigma \subset \P^{N_1}\times \cdots \times \P^{N_m}$ containing $C$ of $H_1$-degree $\Sigma \cdot H_1^2 = \deg Q + (H_1\cdot C)$. And since $\Sigma$ is a $\P^{N_1}$-scroll, we have $H_1H_j\cdot \Sigma = H_j\cdot C$ for $j\neq 1$. Thus, $\Sigma$ has numerical class
\[ (H_1\cdot C + \deg Q)H_1^{N_1-2}H_2^{N_2}\cdots H_m^{N_m} + (H_2\cdot C)H_1^{N_1-1}H_2^{N_2-1}\cdots H_m^{N_m} + \cdots + (H_m\cdot C)H_1^{N_1-1}H_2^{N_2}\cdots H_m^{N_m-1}. \]

Since $A$ is projectively normal, there exists a degree $(d_1, \hdots , d_m)$ hypersurface $Z\subset \P^{N_1}\times \cdots \times \P^{N_m}$ such that $X = A\cap Z$. We divide the proof into the cases $\Sigma \subset Z$ and $\Sigma \not \subset Z$.

First suppose  that $\Sigma \subset Z$. We will prove in Theorem~\ref{linestheorem} that $X$ does not contain $\P^{N_1}$-lines. Thus, we must have $\Sigma \not \subset A$. As rational homogeneous varieties in projective space are cut out by quadrics (\textit{cf.} \cite[Theorem 3.11 and Remark 3.12]{Ram87}), by the Segre embedding, we can express $A$ as cut out by varieties of degree $(2,\hdots ,2)$ in $\P^{N_1}\times \cdots \times \P^{N_m}$.

Then, there exists a hypersurface $F$ of class $2H_1 + \cdots + 2H_m$ such that $\Sigma \not \subset F$. Intersecting the classes of~$\Sigma$ and $F$, we have
\[ ([\Sigma]\cdot [F])\cdot H_1 = 2(H_1\cdot C + \deg Q) + 2(H_2\cdot C) + \cdots + 2(H_m\cdot C). \]

Since $C\subset \Sigma\cap F$, it follows that $([\Sigma]\cdot [F])\cdot H_1\ge H_1\cdot C$, so
\[ 2(H_1\cdot C + \deg Q) + 2(H_2\cdot C) + \cdots + 2(H_m\cdot C)\ge H_1\cdot C; \]
hence
\begin{equation}\label{first_deg_Q}
    \deg Q\ge -\frac{1}{2}(H_1\cdot C) - (H_2\cdot C) - \cdots - (H_m\cdot C).
\end{equation}

Now, a proof similar to that of Lemma~\ref{surjectioninequality} for the surjection~\eqref{quot_map} yields the inequality
\[ \deg N_{C/X}\ge \deg Q + \deg \beta(M_{H_1}^{\oplus D-3})\ge \deg Q -(D-3)(H_1\cdot C). \]

Thus,
\[ \deg N_{C/X} = 2g(C) - 2 - (K_X\cdot C)\ge\deg Q -(D-3)(H_1\cdot C); \]
hence,
\begin{align*}
    2g(C)-2\ge & (K_X\cdot C) - (D-3)(H_1\cdot C) + \deg Q \\
    = & \sum_{i=1}^m(a_i+d_i)(H_i\cdot C) - (D-3)(H_1\cdot C) + \deg Q, 
\end{align*}
and by Inequality~\eqref{first_deg_Q}, we have
\[ 2g(C)-2\ge \left (a_1+d_1-D+2+\frac{1}{2}\right )(H_1\cdot C) + \sum_{i=2}^m(a_i+d_i-1)(H_i\cdot C). \]
Since, by hypothesis, $d_i\ge D-a_i-2$ for all $1\le i\le m$ and $D\ge 4$, we get
\[ 2g(C)-2\ge \frac{1}{2}(H_1\cdot C) + \sum_{i=2}^m(H_i\cdot C). \]

Therefore, in this case we have $2g(C)-2\ge \epsilon \sum_{i=1}^m(H_i\cdot C)$ for $\epsilon = \frac{1}{2}$; hence $X$ is algebraically hyperbolic.

Now suppose  instead that $\Sigma \not \subset Z$. We intersect $\Sigma$ and $Z$ to get
\[ ([\Sigma]\cdot [Z])\cdot H_1 = d_1(H_1\cdot C + \deg Q) + d_2(H_2\cdot C) + \cdots + d_m(H_m\cdot C). \]

As $C\subset \Sigma \cap Z$, we have $([\Sigma]\cdot [Z])\cdot H_1\ge H_1\cdot C$, thus
\[ d_1(H_1\cdot C + \deg Q) + d_2(H_2\cdot C) + \cdots + d_m(H_m\cdot C)\ge H_1\cdot C, \]
so
\begin{equation}\label{second_deg_Q}
    \deg Q\ge \left (\frac{1}{d_1}-1\right )(H_1\cdot C) - \frac{d_2}{d_1}(H_2\cdot C) - \cdots - \frac{d_m}{d_1}(H_m\cdot C).
\end{equation}

And as before, we have
\[ 2g(C) - 2\ge \sum_{i=1}^m(a_i+d_i)(H_i\cdot C) - (D-3)(H_1\cdot C) + \deg Q, \]
so by Inequality~\eqref{second_deg_Q}, we get
\[ 2g(C)-2\ge \left (a_1+d_1-D+2+\frac{1}{d_1}\right )(H_1\cdot C) + \sum_{i=2}^m\left (a_i+d_i-\frac{d_i}{d_1}\right )(H_i\cdot C). \]

Therefore, it suffices to have $a_1+d_1\ge D-2$ and $a_i+d_i-\frac{d_i}{d_1}>\epsilon$ for some $\epsilon > 0$ and $2\le i\le m$. The first inequality is true by hypothesis. We now check the second inequality.

First, we observe that $a_i\le -2$ for all $1\le i\le m$. Indeed, if $a_i\ge -1$, then by Lemma~\ref{dim_lines_A}, the dimension of the space of $\P^{N_i}$-lines in $A$ would be $\dim A-a_i-3\le \dim A-2$, but then the lines could not cover $A$, which is not the case for a rational homogeneous variety (\textit{cf.} \cite[Theorem 1.5 and the following paragraph]{Man20}). And since $d_i\ge D-a_i-2$, this implies that $d_i\ge 4$ for all $1\le i\le m$.

On the other hand, an $n$-dimensional projective manifold $Y$ with an ample divisor $H$ satisfies the inequality $K_Y \ge -(n+1)H$, with equality when $X$ is isomorphic to $\P^n$ (\textit{cf.}~\cite{KO73}.) Thus, looking at each projection $\pi_i\colon A\to \P^{N_i}$, we obtain that $a_i \ge -(D+1)$ for each $1\le i\le m$.

Now, rearrange the inequality we want to prove, $a_i+d_i-\frac{d_i}{d_1}>\epsilon$, as 
\[ d_i > (\epsilon - a_i)\frac{d_1}{d_1-1}. \]

Since $d_1\ge 4$, we have $\frac{d_1}{d_1-1}\le \frac{4}{3}$, and as $d_i\ge D-a_i-2$, it suffices to show that
\[ D-2-a_i > (\epsilon - a_i)\frac{4}{3}, \]
or equivalently
\[ D-2\ge -\frac{a_i}{3} + \epsilon ', \]
for $\epsilon ' = 4\epsilon /3$.

Now, as $a_i\ge -(D+1)$, we have $-\frac{a_i}{3} + \epsilon '\le \frac{D+1}{3} + \epsilon '$, so we only need to show that
\[ D-2\ge \frac{D+1}{3} + \epsilon '. \]

Rearranging terms, this is equivalent to
\[ D\ge \frac{7}{2} + \frac{3\epsilon '}{2}, \]
which is true since $D\ge 4$.

Therefore, in all cases we get $2g(C)-2\ge \epsilon \sum_{i=1}^m(H_i\cdot C)$ for some $\epsilon > 0$; hence $X$ is algebraically hyperbolic.
\end{proof}

\section{Lines in hypersurfaces of homogeneous varieties}\label{sec4}

In this section, we obtain the degrees for which the general hypersurface in $A$ contains lines. In particular, we get a bound for when $X$ is not algebraically hyperbolic. First, we collect a lemma characterizing certain Schubert classes in the Grassmannian of lines, based on Liu's work \cite{Liu22}.

\begin{lemma}[\textit{cf.} {\cite[Proposition 3.1]{Liu22}}]\label{Schublemma}
    Let $d\ge 2$, $m>0$, and let $X\subset G(2,N)$ be a $($possibly reducible$)$ variety with class $[X] = m\cdot \sigma_{N-2, N-2-(d+1)}$ in the Chow ring of the Grassmannian of lines $G(2,N)$. Then, the lines parametrized by $X$ pass through a finite set of fixed points in $\P^{N-1}$.
\end{lemma}

\begin{proof}
  Consider the dual Grassmannian $G(N-2,N)$. The class of the dual variety of 
  $[X]=m\cdot \sigma_{N-2, N-2-(d+1)}$ in $G(N-2,N)$ is $[X^*]=m\cdot \sigma_{2^{N-2-(d+1)}, 1^{d+1}}$.

    By Pieri's formula \cite[Proposition 4.9]{3264}, we have $[X^*]\cdot \sigma_2 = 0$. Since $\sigma_2$ parametrizes copies of $\P^{N-3}$ through a point in $\P^{N-1}$, it follows that the general point of $\P^{N-1}$ is not contained in any of the spaces $\P^{N-3}$ parametrized by $X^*$. Hence, the spaces $\P^{N-3}$ parametrized by each component of $X^*$ sweep out a subvariety $T$ of dimension $N-2$ of $\P^{N-1}$. We claim that $T$ is a hyperplane, $T\cong \P^{N-2}$. To see this, consider the universal family
    \[ \mathcal{U} = \{ (x, p) \ | \ x\in X^*, \ p \text{ lies in the } \P^{N-3} \text{ parametrized by } x \}\subset X^*\times T. \]

    The fibers of the projection $\mathcal{U}\to X^*$ are $\P^{N-3}$, so $\dim \mathcal{U}=N-3+\dim X^*$. Thus, by the projection $\mathcal{U}\to T$, there is a $(\dim X^*-1)$-parameter family of copies of $\P^{N-3}$ through each point $p$ of $T$. Since $d\ge 2$, we have $\dim X^*-1\ge 2$, so we get a two-parameter family of copies of $\P^{N-3}$ through each point $p$ of $T$. Each of these copies of $\P^{N-3}$ is in the tangent space $T_p(T)\cong \P^{N-2}$; thus by dimension, $T_p(T)=\P^{N-2}=T$.

    Therefore, there exists one $\P^{N-2}\subset \P^{N-1}$ such that every $\P^{N-3}$ parametrized by a component of $X^*$ is contained in $\P^{N-2}$. Taking back duals, every line parametrized by $X$ passes through a fixed point in $\P^{N-1}$. Taking the union for all components of $X$, there is a finite set of fixed points in $\P^{N-1}$ such that every line parametrized by $X$ passes through one of these points.
\end{proof}

\begin{theorem}\label{linestheorem}
    Let $A$ be a rational homogeneous variety as in Main Setting~\ref{main_setting}. Let $X$ be a general hypersurface of degree $(d_1, \hdots , d_m)$ in $A$.
    \begin{itemize}
        \item If $d_i\le \dim A - a_i - 4$ for some $1\le i\le m$, then $X$ contains $\P^{N_i}$-lines. In particular, $X$ is not algebraically hyperbolic.
        \item If $d_i > \dim A - a_i - 4$ for some $1\le i\le m$, then $X$ does not contains $\P^{N_i}$-lines.
    \end{itemize}
\end{theorem}
\begin{proof}
    Set $D = \dim A$. To simplify the notation, let us assume that $d_1\le D-a_1-4$. Then, we want to show that $X$ contains $\P^{N_1}$-lines.

    Since $A$ is projectively normal, there is a degree $(d_1,\hdots,d_m)$ hypersurface $Z\subset \P^{N_1}\times \cdots \times \P^{N_m}$ such that $X = A\cap Z$. Let $F_1^{N_1}(Z)\subset G(2,N_1+1)\times \P^{N_2}\times \cdots \times \P^{N_m}$ be the Fano scheme of $\P^{N_1}$-lines in $Z$, and let $F_1^{N_1}(A)$ be the Fano scheme of $\P^{N_1}$-lines in $A$. Denote by $[F_1^{N_1}(Z)]$ and $[F_1^{N_1}(A)]$ their classes in the Chow ring of $G(2,N_1+1)\times \P^{N_2}\times \cdots \times \P^{N_m}$. To show that $X$ contains lines, we are going to prove that the intersection class $[F_1^{N_1}(Z)]\cdot [F_1^{N_1}(A)]$ is not zero.

    By the definition of degree, a general $\P^{N_1}$ section of $Z$ is a general degree $d_1$ hypersurface in $\P^{N_1}$. By \cite[Theorem 6.34]{3264}, the scheme of lines in a general degree $d_1$ hypersurface has codimension $d_1+1$ in $G(2,N_1+1)$. Thus, $F_1^{N_1}(Z)$ has codimension $d_1+1$ in $G(2,N_1+1)\times \P^{N_2}\times \cdots \times \P^{N_m}$. On the other hand, $\dim F_1^{N_1}(A) = D - a_1 - 3$ by Lemma~\ref{dim_lines_A}. Therefore, it suffices to prove the theorem for $D-a_1-3=d_1+1$, or equivalently $d_1=D-a_1-4$.
        
    The intersection of $[F_1^{N_1}(Z)]$ with $H_2^{N_2}\cdots H_m^{N_m}$ is the class of lines in a general $\P^{N_1}$ section of $Z$, which is a general degree $d_1$ hypersurface in $\P^{N_1}$. So, by \cite[Proposition 6.4]{3264}, we can compute the class $[F_1^{N_1}(Z)]$:
    \[ \left[F_1^{N_1}(Z)\right] = c_{d_1+1}(\Sym^{d_1} S^*), \]
    where $S$ is the universal subbundle of $G(2, N_1+1)$.

    And since $F_1^{N_1}(A)$ is a space of $\P^{N_1}$-lines, its class $[F_1^{N_1}(A)]$ is also generated by classes of $G(2,N_1+1)$.

    By the Whitney formula and the splitting principle, see \cite[Chapter 5]{3264}, we can write
    \[ 1 + \sigma_1 + \sigma_{1,1} = c(S^*) = (1+\alpha)(1+\beta) \]
    for $\alpha + \beta = \sigma_1$ and $\alpha \beta = \sigma_{1,1}$. Thus,
    \[ c(\Sym^{d_1} S^*) = (1+d_1\alpha)[1+(d_1-1)\alpha +\beta]\cdots [1+\alpha+(d_1-1)\beta](1+d_1\beta). \]
    We are interested in the top Chern class
    \[ c_{d_1+1}(\Sym^{d_1} S^*) = (d_1\alpha)[(d_1-1)\alpha+\beta]\cdots [\alpha+(d_1-1)\beta](d_1\beta). \]
    
    Suppose that $d_1$ is even (the case when $d_1$ is odd is similar), and rearrange the expression above as
    \begin{flalign*}
    & \left[F_1^{N_1}(Z)\right] = c_{d_1+1}(\Sym^{d_1} S^*) \\
    & = (d_1\alpha)(d_1\beta)[(d_1-1)\alpha+\beta][\alpha+(d_1-1)\beta]\cdots [i\alpha+(d_1-i)\beta][(d_1-i)\alpha+i\beta]\cdots \left (\frac{d_1}{2}\alpha+\frac{d_1}{2}\beta\right ) \\
    & = d_1^2(\alpha\beta)[(d_1-1)(\alpha+\beta)^2+(d_1-2)^2(\alpha\beta)]\cdots [i(d_1-i)(\alpha+\beta)^2+(d_1-2i)^2(\alpha\beta)]\cdots \frac{d_1}{2}(\alpha+\beta) \\
    & = d_1^2\sigma_{1,1}[(d_1-1)\sigma_1^2+(d_1-2)^2\sigma_{1,1}]\cdots [i(d_1-i)\sigma_1^2+(d_1-2i)^2\sigma_{1,1}]\cdots \frac{d_1}{2}\sigma_1.
    \end{flalign*}
    
    Notice that, after the terms are expanded, all coefficients are non-negative and the term $\sigma_{1,1}\sigma_1^{d_1-1}$ appears. Recall that these Schubert classes multiply by the rule $\sigma_1\cdot \sigma_{k,l} = \sigma_{k+1,l} + \sigma_{k,l+1}$ and $\sigma_{1,1}\cdot \sigma_{k,l} = \sigma_{k+1,l+1}$. Thus, we can see that, in the expansion of $\sigma_1^{d_1-1}$, all Schubert classes $\sigma_{i,j}$ with $i+j=d_1-1$ appear with positive coefficients. It follows that, in the expansion of $[F_1^{N_1}(Z)]$, all classes $\sigma_{i,j}$ with $i+j=d_1+1$ and $i,j\ge 1$ appear with positive coefficients; that is, the only degree $d_1+1$ class not in the expansion of $[F_1^{N_1}(Z)]$ is $\sigma_{d_1+1,0}$. Since $\sigma_{N-1,N-1-(d_1+1)}$ is the only Schubert class of dimension $d_1+1$ whose intersection with $\sigma_{d_1+1,0}$ in $G(2,N+1)$ is non-zero, it suffices to show that $[F_1^{N_1}(A)]$ is not a multiple of $\sigma_{N-1,N-1-(d_1+1)}$. But, by Lemma~\ref{Schublemma}, if $[F_1^{N_1}(A)]$ is a multiple of $\sigma_{N-1,N-1-(d_1+1)}$, then the lines of a general $\P^{N_1}$ section of $A$ pass through a finite set of points, which cannot happen since $A$ is homogeneous. Therefore, the product $[F_1^{N_1}(Z)]\cdot [F_1^{N_1}(A)]$ is not zero; hence there are $\P^{N_1}$-lines contained in $X$.

    On the other hand, if $d_i > D - a_i - 4$ for some $1\le i\le m$, then the dimension of $F_1^{N_1}(A)$ is less than the codimension of $F_1^{N_1}(Z)$ in $G(2, N_i+1)$. Since $X$ is general, $F_1^{N_1}(Z)$ is general, and therefore the intersection $F_1^{N_1}(A)\cap F_1^{N_1}(Z)$ is empty.
\end{proof}

By Theorems~\ref{scroll_method} and~\ref{linestheorem}, we have our main result.

\begin{theorem}\label{oTeorema}
    Let $A$ be a rational homogeneous variety as in  Main Setting~\ref{main_setting}.
    \begin{itemize}
        \item If\, $\dim A\ge 4$ and $d_i\ge \dim A-a_i-2$ for all $1\le i\le m$, then a very general hypersurface $X$ of degree $(d_1, \hdots, d_m)$ is algebraically hyperbolic.
        \item If\, $d_i\le \dim A - a_i - 4$ for some $1\le i\le m$, then a general hypersurface $X$ of degree $(d_1, \hdots, d_m)$ contains $\P^{N_i}$-lines. In particular, $X$ is not algebraically hyperbolic.
    \end{itemize}
\end{theorem}

This leaves open the cases with $d_i\ge \dim A - a_i - 3$ for all $1\le i\le m$ and $d_i = \dim A-a_i-3$ for some $1\le i\le m$. When $\dim A\le 4$, there are examples when this bound fails to imply algebraic hyperbolicity. 

\begin{example}[\textit{cf.} {\cite[Lemma 4.2]{Y22}}]\label{Example_2_2}
    A very general hypersurface $X\subset \P^2\times \P^2$ of degree $(d_1,d_2)$ with $d_1=4$ or $d_2=4$ contains an elliptic curve. In particular, it is not algebraically hyperbolic.
\end{example}

The same proof also works for $\P^2\times \P^1\times \P^1$.

\begin{example}\label{Example_2_1_1}
    A very general hypersurface $X\subset \P^2\times \P^1\times \P^1$ of degree $(d_1,d_2,d_3)$ with $d_1=4$ contains an elliptic curve. In particular, it is not algebraically hyperbolic.
\end{example}

The bound $d_i\ge \dim A - a_i - 3$ for all $i$ also fails to imply algebraic hyperbolicity in the cases $\P^1\times \P^1\times \P^1$ and $\P^2\times \P^1$, as proved in \cite{CR19}. For $A = \P^4$, the case of sextic threefolds remains open.
However, when $\dim A\ge 5$, the bound does imply algebraic hyperbolicity in the cases $A = \P^n$ and $A=\P^m\times \P^n$, as proved in \cite{Y22}. We conjecture the same holds for rational homogeneous varieties.

\begin{conjecture}\label{dim_5_conjecture}
    If\, $\dim A\ge 5$ and $d_i\ge \dim A - a_i - 3$ for all $i$, then is $X$ algebraically hyperbolic.
\end{conjecture}

\section{Examples}\label{sec5}

In this section, we apply Theorem~\ref{oTeorema} to explicit examples of homogeneous varieties. For more details on these varieties, we refer to \cite{IMPANGA}. The projective normality hypothesis follows from \cite{RR85}. The section-dominating collections of line bundles follow from Examples~\ref{section-dominating} and~\ref{section-dom-products}. We remark that we can also apply the theorem to products of these varieties.

\subsection{Grassmannians}

Let $A = G(k,n)$ be the Grassmannian of $k$-dimensional subspaces in an $n$-dimensional vector space. Consider $G(k,n)$ embedded in $\P^N$ via the Pl\"ucker embedding $p\colon G(k,n)\to \P^N$. Let $H$ be the hyperplane section in $G(k,n)$ corresponding to $\cO_{G(k,n)}(1) = p^*\cO_{\P^N}(1)$. Let $\mathcal{E}$ be the line bundle $dH$ for some degree $d>0$. The dimension of the Grassmannian is $\dim G(k,n) = k(n-k)$. The tangent bundle is $T_{G(k,n)}\cong S^*\otimes Q$, where $S$ is the tautological bundle and $Q$ is the quotient bundle of $G(k,n)$. Then a Chern class computation gives
\[ c_1(T_{G(k,n)}) = c_1(S^*\otimes Q) = nH. \]
Thus, the canonical divisor is $K_{G(k,n)} = -nH$.

\begin{theorem}[Grassmannians]
\leavevmode
  \begin{itemize}
    \item If\, $\dim G(k,n) = k(n-k)\ge 4$ and $d\ge k(n-k)+n-2$, then a very general degree $d$ hypersurface $X\subset G(k,n)$ is algebraically hyperbolic.
    \item If\, $d\le k(n-k)+n-4$, then a general degree $d$ hypersurface $X$ contains a line. In particular, $X$ is not algebraically hyperbolic.
\end{itemize}
\end{theorem}

\subsection{Products of Grassmannians}

Let $A$ be the product of Grassmannians $A = \prod_{i=1}^m G(k_i,n_i)$. We can see that $A\subset \P^{N_1}\times \cdots \times \P^{N_m}$ via the Pl\"ucker embedding for each factor. Let $H_i$ be the pullback of the hyperplane section via each projection $A\to G(k_i,n_i)$. For $d_i\ge 0$, let $\mathcal{E} = \sum _{i=1}^m d_iH_i$. The dimension of $A$ is $\sum_{i=1}^m k_i(n_i-k_i)$. The canonical divisor of $A$ is $K_A = \sum _{i=1}^m (-n_i)H_i$.

\begin{theorem}[Products of Grassmannians]
\leavevmode
  \begin{itemize}
    \item If\, $\dim A = \sum_{i=1}^m k_i(n_i-k_i)\ge 4$ and $d_i\ge \left (\sum_{j=1}^m k_j(n_j-k_j)\right ) + n_i - 2$ for all $1\le i\le m$, then a very general hypersurface of\, $\prod_{i=1}^m G(k_i,n_i)$ of degree $(d_1,\hdots,d_m)$ is algebraically hyperbolic.
    \item If\, $d_i\le \left (\sum_{j=1}^m k_j(n_j-k_j)\right ) + n_i - 4$ for some $1\le i\le m$, then a general hypersurface of degree $(d_1,\hdots ,d_m)$ of\, $\prod_{i=1}^m G(k_i,n_i)$ contains a line. In particular, it is not algebraically hyperbolic.
\end{itemize}
\end{theorem}

\subsection{Orthogonal Grassmannians}

Let $A$ be the orthogonal Grassmannian $\OG(k,n)$. Similarly, let $H$ be the hyperplane section and $\mathcal{E}=dH$ for $d>0$. We have $\dim \OG(k,n) = \frac{k(2n-3k-1)}{2}$. The orthogonal Grassmannian is the zero locus in $G(k,n)$ of a global section of $\Sym ^2(S^*)$, and we have an exact sequence of tangent bundles
\[ 0\longrightarrow T_{\OG(k,n)}\longrightarrow S^*\otimes Q\longrightarrow \Sym^2(S^*)\longrightarrow 0. \]
So, by a Chern class computation, we have
\[ c_1(T_{\OG(k,n)}) = c_1(S^*\otimes Q) - c_1(\Sym^2 (S^*)) = (n-k-1)H. \]

Thus, the canonical divisor of $\OG(k,n)$ is $K_{\OG} = -(n-k-1)H$.

\begin{theorem}[Orthogonal Grassmannians]
  \leavevmode
    \begin{itemize}
        \item If\, $\dim \OG(k,n)=\frac{k(2n-3k-1)}{2}\ge 4$ and $d\ge \frac{k(2n-3k-1)}{2}+n-k-3$, then a very general hypersurface of $\OG(k,n)$ of degree $d$ is algebraically hyperbolic.
        \item If\, $d\le \frac{k(2n-3k-1)}{2}+n-k-5$, then a general hypersurface of degree $d$ of\, $\OG(k,n)$ contains a line. In particular, it is not algebraically hyperbolic.
    \end{itemize}
\end{theorem}

\subsection{Symplectic Grassmannians}

Let $A = \SG(k,n)$ be a symplectic Grassmannian. Let $H$ by the hyperplane section, and $\mathcal{E} = dH$ for $d>0$. In this case, $\dim \SG(k,n) = \frac{k(2n-3k+1)}{2}$. Similarly, we have an exact sequence
\[ 0\longrightarrow T_{\SG(k,n)}\longrightarrow S^*\otimes Q\longrightarrow \wedge^2(S^*)\longrightarrow 0, \]
so we can compute
\[ c_1(T_{\SG(k,n)}) = c_1(S^*\otimes Q) - c_1(\wedge^2 (S^*)) = (n-k+1)H. \]

Thus, the canonical divisor is $K_{\SG(k,n)} = -(n-k+1)H$.

\begin{theorem}[Symplectic Grassmannians]
  \leavevmode
    \begin{itemize}
        \item If\, $\dim A=\frac{k(2n-3k+1)}{2}\ge 4$ and $d\ge \frac{k(2n-3k+1)}{2}+n-k-1$, then a very general degree $d$ hypersurface of\, $\SG(k,n)$ is algebraically hyperbolic.
        \item If\, $d\le \frac{k(2n-3k+1)}{2}+n-k-3$, then a general degree $d$ hypersurface of $\SG(k,n)$ contains a line. In particular, it is not algebraically hyperbolic.
    \end{itemize}
\end{theorem}

\subsection{Flag varieties}

Let $A = F(k_1, \hdots, k_m; n)$ be a flag variety. The projections
\[ F(k_1, \hdots, k_m; n)\lra G(k_1,n)\times \cdots \times G(k_m,n) \]
give, via pullback, the hyperplane classes $H_1, \hdots , H_m$ that generate the Picard group of $A$. To fix notation, let $k_0 = 0$ and $k_{m+1} = n$. The dimension of $A$ is $\sum_{i=1}^{m}k_{i}(k_{i+1}-k_i)$. The tangent bundle of $A$ is isomorphic to $\bigoplus_{i=1}^m S_{k_i}^*\otimes (S_{k_{i+1}}/S_{k_i})$, where $S_{k_i}$ is the rank $k_i$ tautological bundle. Then, by a Chern class computation,
\[ c_1(F(k_1, \hdots, k_m; n)) = c_1\left (\bigoplus_{i=1}^m S_{k_i}^*\otimes (S_{k_{i+1}}/S_{k_i})\right ) = \sum_{i=1}^m (k_{i+1}-k_{i-1})H_i. \]

Therefore, the canonical divisor of $F(k_1, \hdots, k_m; n)$ is
\[K_A = - \sum_{i=1}^{m}(k_{i+1}-k_{i-1})H_{i} = -k_2H_1 - (k_3-k_1)H_2 - \cdots - (n-k_{m-1})H_m.\]

\begin{theorem}[Flag varieties] \leavevmode
    \begin{itemize}
        \item If\, $\dim A = \sum_{i=1}^{m}k_{i}(k_{i+1}-k_i)\ge 4$ and $d_{i}\ge \left (\sum_{j=1}^{m}k_{j}(k_{j+1}-k_j)\right ) + k_{i+1} - k_{i-1} - 2$ for all $1\le i\le m$, then a very general hypersurface of $F(k_1, \hdots, k_m; n)$ of degree $(d_1, \hdots, d_m)$ is algebraically hyperbolic.
        \item If\, $d_{i}\le \left (\sum_{j=1}^{m}k_{j}(k_{j+1}-k_j)\right ) + k_{i+1} - k_{i-1} - 4$ for some $i$, then a general hypersurface of $F(k_1, \hdots, k_m; n)$ of degree $(d_1, \hdots, d_m)$ contains a line. In particular, it is not algebraically hyperbolic.
    \end{itemize}    
\end{theorem}



\begin{thebibliography}{Ram87+++}

\bibitem[Bro78]{Br78}
R.~Brody, \emph{Compact Manifolds and Hyperbolicity}, Trans.\ Amer.\ Math.\ Soc.\ \textbf{235} (1978), 213--219, \doi{10.2307/1998216}.

\bibitem[Cle86]{Cl86}
H.~Clemens, \emph{Curves on generic hypersurfaces}, Ann.\ Sci.\ \'Ecole Norm.\
  Sup.~(4) \textbf{19} (1986), no.~4, 629--636, 
  \doi{10.24033/asens.1521}.

\bibitem[Cle03]{Cl03}
\bysame, \emph{Lower Bounds on Genera of Subvarieties of Generic
  Hypersurfaces}, Comm.\ Algebra \textbf{31} (2003), no.~8,
  3673--3711, \doi{10.1081/agb-120022438}.

\bibitem[CR04]{ClR04}
H.~Clemens and Z.~Ran, \emph{Twisted genus bounds for subvarieties of generic
  hypersurfaces}, Amer.~J.\ Math.\ \textbf{126} (2004), no.~1,
  89--120, \doi{10.1353/ajm.2004.0003}.

\bibitem[Cos18]{IMPANGA}
  I.~Coskun, \emph{Restriction varieties and the rigidity problem}, in: \emph{Schubert varieties, equivariant cohomology and characteristic classes---IMPANGA 15}, pp.~49--96, EMS Ser.\ Congr.\ Rep., Eur.\ Math.\ Soc., Z\"urich, 2018, \doi{10.4171/182-1/4}.

\bibitem[CR19]{CR19a}
I.~Coskun and E.~Riedl, \emph{Algebraic hyperbolicity of the very general
  quintic surface in $\mathbb{P}^3$}, Adv.\ Math.\ \textbf{350} (2019),
  1314--1323,
  \doi{10.1016/j.aim.2019.04.062}.

\bibitem[CR23]{CR19}
  \bysame, \emph{Algebraic hyperbolicity of very general surfaces}, Israel J.~Math.\  \textbf{253} (2023), no.~2, 787--811,
  \doi{10.1007/s11856-022-2379-2}.

\bibitem[Deb01]{HDAG}
O.~Debarre, \emph{Higher-Dimensional Algebraic Geometry}, Universitext,
  Springer-Verlag, New York, 2001, \doi{10.1007/978-1-4757-5406-3}.

\bibitem[Dem97]{De95}
J.~Demailly, \emph{Algebraic criteria for Kobayashi hyperbolic projective
  varieties and jet differentials}, in: \emph{Algebraic geometry---Santa Cruz 1995}, pp.~285--360, Proc.\ Sympos.\ Pure Math., vol.~62, Part~2, Amer.\ Math.\ Soc., Providence, RI, 1997, 
  \doi{10.1090/pspum/062.2/1492539}.

\bibitem[Dem20]{De20}
\bysame, \emph{Recent results on the Kobayashi and the Green--Griffiths--Lang
  conjectures}, Jpn.~J.\ Math. \textbf{15} (2020), no.~1, 1--120,
\doi{10.1007/s11537-019-1566-3}.

\bibitem[Ein88]{Ei88}
L.~Ein, \emph{Subvarieties of generic complete intersections}, Invent.\ Math.\
  \textbf{94} (1988), no.~1, 163--169, \doi{10.1007/bf01394349}.

\bibitem[Ein91]{Ei91}
\bysame, \emph{Subvarieties of generic complete intersections. II}, Math.\ Ann.\
  \textbf{289} (1991), no.~3, 465--471, \doi{10.1007/bf01446583}.

\bibitem[EH16]{3264}
D.~Eisenbud and J.~Harris,  \emph{3264 and All That: A Second Course in
  Algebraic Geometry}, Cambridge Univ.\ Press, Cambridge, 2016, 
\doi{10.1017/CBO9781139062046}.

\bibitem[GG80]{GG80}
M.~Green and P.~Griffiths, \emph{Two Applications of Algebraic Geometry to
  Entire Holomorphic Mappings}, in: \emph{The Chern Symposium 1979} (Proc.\ Internat.\  Sympos., Berkeley, Calif., 1979), pp.~41--74, Springer-Verlag, New York-Berlin, 1980, \doi{10.1007/978-1-4613-8109-9_4}.

\bibitem[HI21]{HI21}
C.~Haase and N.~Ilten, \emph{Algebraic hyperbolicity for surfaces in toric
  threefolds}, J.~Algebraic Geom.\ \textbf{30} (2021), no.~3,  573--602, \doi{10.1090/jag/770}.

\bibitem[KO73]{KO73}
S.~Kobayashi and T.~Ochiai, \emph{Characterizations of complex projective
  spaces and hyperquadrics}, J.~Math.\ Kyoto Univ.\ \textbf{13} (1973), no.~1,
  31--47, \doi{10.1215/kjm/1250523432}.

\bibitem[Lan86]{L86}
S.~Lang, \emph{Hyperbolic and Diophantine analysis}, Bull.\ Amer.\ Math.\ Soc.\
  (N.S.) \textbf{14} (1986), no.~2, 159--205,
  \doi{10.1090/s0273-0979-1986-15426-1}.

\bibitem[Liu22]{Liu22}
Y.~Liu, \emph{The Rigidity Problem in Orthogonal Grassmannians},
preprint  \arXiv{2210.14540} (2022).

\bibitem[Man20]{Man20}
L.~Manivel, \emph{Topics on the Geometry of Rational Homogeneous Spaces}, Acta
  Math.\ Sin.\ (Engl.\ Ser.) \textbf{36} (2020), no.~8, 851--872,
  \doi{10.1007/s10114-020-9386-1}.

\bibitem[Pac03]{Pa03}
G.~Pacienza, \emph{Rational curves on general projective hypersurfaces},
  J.~Algebraic Geometry \textbf{12} (2003), no.~2, 245--267,
  \doi{10.1090/s1056-3911-02-00328-4}.

\bibitem[Pac04]{Pa04}
\bysame, \emph{Subvarieties of general type on a general projective
  hypersurface}, Trans.\ Amer.\ Math.\ Soc.\ \textbf{356}
(2004), no.~7, 2649--2661, \doi{10.1090/s0002-9947-03-03250-1}.

\bibitem[RR85]{RR85}
S.~Ramanan and A.~Ramanathan, \emph{Projective normality of flag varieties and
  Schubert varieties}, Invent.\ Math.\ \textbf{79} (1985), no.~2, 217--224, \doi{10.1007/bf01388970}.

\bibitem[Ram87]{Ram87}
A.~Ramanathan, \emph{Equations defining Schubert varieties and Frobenius
  splittings of diagonals}, Publ.\ Math.\ Inst.\ Hautes \'Etudes Sci.\
  \textbf{65} (1987), 61--90, \doi{10.1007/bf02698935}.

\bibitem[Rob23]{Rob23}
S.~Robins, \emph{Algebraic hyperbolicity for surfaces in smooth projective
  toric threefolds with Picard rank 2 and 3}, Beitr.\ Algebra Geom.\ \textbf{64}
  (2023), no.~1, 1--27, \doi{10.1007/s13366-021-00612-0}.

\bibitem[Voi96]{Vo96}
C.~Voisin, \emph{On a conjecture of Clemens on rational curves on
  hypersurfaces}, J.~Differential Geom.\ \textbf{44} (1996), no.~1,
  200--213,
  \doi{10.4310/jdg/1214458743}.

\bibitem[Voi98]{Vo98}
\bysame, \emph{A correction: ``On a conjecture of Clemens on rational curves on
  hypersurfaces''}, J.~Differential Geom.\ \textbf{49} (1998), no.~3,
601--611, \doi{10.4310/jdg/1214461112}.

\bibitem[Xu94]{Xu94}
G.~Xu, \emph{Subvarieties of general hypersurfaces in projective space}, J.~Differential Geom.\ \textbf{39} (1994), no.~1, 139--172,
  \doi{10.4310/jdg/1214454680}.

\bibitem[Yeo25]{Y22}
W.~Yeong, \emph{Algebraic hyperbolicity of very general hypersurfaces in
  products of projective spaces}, Israel J.~Math.\ \textbf{266}
  (2025), no.~1, 1--24, \doi{10.1007/s11856-024-2693-y}.

\end{thebibliography}
\end{document}